\documentclass[11pt, reqno, psamsfonts]{amsart}
\pdfoutput=1

\usepackage{amssymb}
\usepackage{amsthm}
\usepackage{amsmath}
\usepackage{latexsym}
\usepackage[T1]{fontenc}
\usepackage[utf8]{inputenc}
\usepackage[russian, french, english]{babel}

\usepackage{graphicx}
\usepackage{wrapfig}
\usepackage{mathtools}
\usepackage{tikz}
\usepackage{amsbsy}
\usepackage[inline]{enumitem}
\usepackage{mathrsfs}
\usetikzlibrary{shapes,snakes}
\usetikzlibrary{arrows.meta}
\usetikzlibrary{decorations.pathmorphing}
\usetikzlibrary{patterns}
\usetikzlibrary{calc}
\usepackage{float}
\usepackage{array}
\usepackage{subcaption}
\usepackage{multicol}
\usepackage{stmaryrd}
\usepackage{cancel} 
\usepackage{algorithm}
\usepackage{algpseudocode}
\makeatletter
\algnewcommand{\LineComment}[1]{\Statex \hskip\ALG@thistlm {\color{gray}\texttt{// #1}}}
\makeatother

\counterwithin{algorithm}{section}
\usepackage[justification=centering, labelfont=bf]{caption}

\usepackage{lmodern}
\usepackage{mathabx}

\usepackage{upgreek}
\usepackage{titlesec}
\usepackage{bbm}
\usepackage[spacing=true,kerning=true,babel=true,tracking=true]{microtype}
\usepackage[shortcuts]{extdash}
\usepackage[foot]{amsaddr}
\usepackage[left=1in,right=1in,top=1in,bottom=1in,bindingoffset=0cm]{geometry}
\usepackage{bm}
\usepackage{centernot}
\usepackage{mdframed}
\usepackage{hyperref}
\usepackage{mathdots}
\usepackage{xspace}
\hypersetup{
    colorlinks=true,
    linkcolor=blue,
    filecolor=magenta,      
    urlcolor=red,
    citecolor=gray,
}
\allowdisplaybreaks

\usepackage[backend=biber, style=alphabetic, sorting=nyt, maxnames=100,backref=true, firstinits=true]{biblatex}
\title{List colorings of $K$-partite $K$-graphs}
\date{}
\author{\lsstyle Abhishek~Dhawan}
\email{adhawan2@illinois.edu}
\address{\textls{\normalfont{}Department of Mathematics, University of Illinois Urbana-Champaign, Urbana, IL, USA}}
\thanks{This research was partially supported by NSF RTG grant DMS-1937241.}

\newtheoremstyle{bfnote}%
{}{}%
{\slshape}{}%
{\bfseries}{\bfseries.}%
{ }%
{\thmname{#1}\thmnumber{ #2}\thmnote{ \ep{\normalfont{}#3}}}

\theoremstyle{bfnote}
\newtheorem{theo}[equation]{Theorem}
\newtheorem*{theo*}{theo}

\newtheorem{Lemma}[equation]{Lemma}
\newtheorem{claim}{Claim}[equation]
\newtheorem{corl}[equation]{Corollary}
\newtheorem{conj}[equation]{Conjecture}

\newtheorem*{corl*}{Corollary}

\newcounter{ForClaims}[section]

\theoremstyle{definition}
\newtheorem{defn}[equation]{Definition}
\newtheorem*{defn*}{Definition}

\newtheorem*{exmp*}{Example}

\theoremstyle{remark}
\newtheorem*{ques*}{Question}
\newtheorem*{remk*}{Remark}

\newcommand*{\myproofname}{Proof}
\newenvironment{claimproof}[1][\myproofname]{\begin{proof}[#1]}{\end{proof}}

\newcommand{\0}{\emptyset}
\newcommand{\set}[1]{\{#1\}}
\newcommand{\N}{{\mathbb{N}}}

\renewcommand{\P}{\mathbb{P}}

\renewcommand{\epsilon}{\varepsilon}
\newcommand{\eps}{\epsilon}
\renewcommand{\phi}{\varphi}
\renewcommand{\theta}{\vartheta}
\renewcommand{\leq}{\leqslant}
\renewcommand{\geq}{\geqslant}
\newcommand{\defeq}{\coloneqq}
\newcommand{\im}{\mathrm{im}}
\newcommand{\bemph}[1]{{\normalfont#1}} 
\newcommand{\ep}[1]{\bemph{(}#1\bemph{)}} 

\newcommand{\emphdef}[1]{\textbf{\textit{{#1}}}}
\newcommand{\pto}{\dashrightarrow}
\newcommand{\emphd}[1]{\emphdef{#1}}

\newcommand{\LLL}{\text{Lov\'asz Local Lemma}}
\numberwithin{equation}{section}

\newcommand{\bbone}{\mathbbm{1}}

\titleformat{\subsection}[block]{\bfseries}{\thesubsection.}{1ex}{}
\titleformat{\subsubsection}[runin]{\itshape}{\bfseries\upshape\thesubsubsection.}{1ex}{}[.---]

\titleformat{\section}[block]{\scshape\filcenter}{\thesection.}{1ex}{}

\titlespacing*{\section}{0pt}{*3}{*1}
\titlespacing*{\subsection}{0pt}{*3}{*1}
\titlespacing*{\subsubsection}{0pt}{*1.5}{*0}

\renewbibmacro{in:}{}

\renewbibmacro*{volume+number+eid}{%
	\printfield{volume}%
	\setunit*{\addnbspace}
	\printfield{number}%
	\setunit{\addcomma\space}%
	\printfield{eid}}

\DeclareFieldFormat[article]{volume}{\textbf{#1}\space}
\DeclareFieldFormat[article]{number}{\mkbibparens{#1}}

\DeclareFieldFormat{journaltitle}{#1,}
\DeclareFieldFormat[thesis]{title}{\mkbibemph{#1}\addperiod}
\DeclareFieldFormat[article, unpublished, thesis]{title}{\mkbibemph{#1},}
\DeclareFieldFormat[book]{title}{\mkbibemph{#1}\addperiod}
\DeclareFieldFormat[unpublished]{howpublished}{#1, }

\DeclareFieldFormat{pages}{#1}

\DeclareFieldFormat[article]{series}{Ser.~#1\addcomma}

\setlist{topsep=3pt,itemsep=3pt}

\begin{document}

\vspace*{0pt}

\maketitle

\begin{abstract}
    A $k$-uniform hypergraph (or $k$-graph) $H = (V, E)$ is $k$-partite if $V$ can be partitioned into $k$ sets $V_1, \ldots, V_k$ such that each edge in $E$ contains precisely one vertex from each $V_i$. In this note, we consider list colorings for such hypergraphs. We show that for any $\varepsilon > 0$ if each vertex $v \in V(H)$ is assigned a list of size $|L(v)| \geq \left((k-1+\varepsilon)\Delta/\log \Delta\right)^{1/(k-1)}$, then $H$ admits a proper $L$-coloring, provided $\Delta$ is sufficiently large. Up to a constant factor, this matches the bound on the chromatic number of simple $k$-graphs shown by Frieze and Mubayi, and that on the list chromatic number of triangle free $k$-graphs shown by Li and Postle. Our results hold in the more general setting of ``color-degree'' as has been considered for graphs. Furthermore, we establish a number of asymmetric statements matching results of Alon, Cambie, and Kang for bipartite graphs.
\end{abstract}

\section{Introduction}\label{sec:intro}

All hypergraphs considered are finite and undirected.
A $k$-uniform hypergraph (or $k$-graph) is an ordered pair $H = (V, E)$ where $E$ is a collection of $k$-element subsets of $V$.
We say $H$ is \textbf{\textit{$k$-partite}} if there exists a partition $V = V_1\cup \cdots \cup V_k$ such that each edge $e \in E$ contains precisely one vertex from each set $V_i$.
A \textbf{\textit{proper $q$-coloring}} of $H$ is an assignment of the integers $1,\, \ldots, \,q$ to the vertices of $H$ such that at least two distinct integers appear on the vertices of each edge.
The minimum number of colors required for a proper coloring is the \textit{\textbf{chromatic number}} of $H$ (denoted $\chi(H)$).
A \textbf{\textit{list assignment for $H$}} is a function $L\,:\,V(H) \to 2^\N$.
Given a list assignment $L$ for $H$, a \textbf{\textit{proper $L$-coloring}} of $H$ is a proper coloring where each vertex receives a color from its list.
The \textit{\textbf{list chromatic number}} (denoted $\chi_\ell(H)$) is the minimum $q$ such that $H$ admits a list coloring whenever $|L(v)| \geq q$ for each $v$.
For $q \in \N$, we let $[q] \defeq \set{1, \ldots, q}$.
For a set $S$ and an element $x \in S$, we let $S - x$ denote the set $S \setminus \set{x}$.
Let $H = (V, E)$ be an undirected $k$-graph.
For each $v \in V$, we let $E_H(v)$ denote the edges containing $v$, $N_H(v)$ denote the set of vertices contained in the edges in $E_H(v)$ apart from $v$ itself, $\deg_H(v) \defeq |E_H(v)|$, and $\Delta(H) \defeq \max_{u\in V}\deg_H(u)$.
Furthermore, for $S \subseteq V$, we let $E_H(S)$ denote the edges $e$ for which $S \subseteq e$, and let $\deg_H(S) \defeq |E_H(S)|$.
Finally, a \textbf{\textit{hypermatching}} is a hypergraph $H$ satisfying $\Delta(H) \leq 1$, and an \textbf{\textit{independent set}} is a subset of vertices containing no edges.

The problem of determining $\chi(G)$ for graphs $(k=2)$ has a long and rich history. 
Brooks provided the first bound in terms of the maximum degree \cite{brooks1941colouring}.
He showed that $\chi(G) \leq \Delta(G)$ unless $G$ is complete or an odd cycle (in this case $\chi(G) = \Delta(G) + 1$).
Reed improved upon this showing that $\chi(G) \leq \Delta(G) - 1$ for sufficiently large $\Delta(G)$, provided $G$ does not contain a clique of size $\Delta(G)$ \cite{reed1999strengthening}.
For $k > 2$, a simple \hyperref[theo:LLL]{\LLL} argument shows that $\chi(H) = O\left(\Delta(H)^{1/(k-1)}\right)$ for all $k$-graphs $H$.
A natural question to consider is the following: under what structural constraints can we get better bounds for $\chi(H)$?
Forbidding a specific subgraph $F$ has led to improved bounds in the case where $k = 2$.
For $F = K_3$, one can show $\chi(G) = O\left(\Delta(G)/\log\Delta(G)\right)$ \cite{Joh_triangle, PS15}.
The constant factor has been reduced to $1+\eps$ \cite{Molloy,bernshteyn2019johansson}, which is optimal up to a factor of $2$ \cite{BollobasIndependence}.
There are a number of results for other graphs $F$, which may be of interest to the reader \cite{AKSConjecture, DKPS, anderson2022coloring, anderson2023colouring}.

Analogous problems have been investigated for $k > 2$.
Frieze and Mubayi studied the chromatic number of \textit{\textbf{simple hypergraphs}}.
A hypergraph $H$ is simple if $\deg_H(S) \leq 1$ for any $S \subseteq V(H)$ satisfying $|S| \geq 2$.
They proved the following bound:

\begin{theo}[\cite{frieze2013coloring}]\label{theo:FM}
    For $k \geq 3$, the following holds for sufficiently large $\Delta \in \N$.
    Let $H$ be a simple $k$-graph of maximum degree at most $\Delta$.
    Then for some constant $c \defeq c(k) > 0$, we have
    \[\chi(H) \leq c\left(\frac{\Delta}{\log \Delta}\right)^{1/(k-1)}.\]
\end{theo}

For $k = 3$, Cooper and Mubayi extended this result to triangle-free hypergraphs \cite{cooper2016coloring}.
A \textbf{\textit{triangle}} in a hypergraph is a set of three pairwise intersecting edges with no common vertex. 
A non-uniform hypergraph $H = (V, E)$ has \textbf{\textit{rank}} $k$ if every edge $e \in E$ satisfies $|e| \leq k$.
In such a hypergraph, the \textbf{\textit{$\ell$-degree}} of a vertex $v$ is the number of edges of size $\ell$ containing $v$.
We let $\Delta_\ell(H)$ denote the maximum $\ell$-degree of a vertex in $H$.
Li and Postle recently extended Cooper and Mubayi's result to all rank $k$ hypergraphs for $k \geq 3$.

\begin{theo}[\cite{li2022chromatic}]\label{theo:LP}
    For $k \geq 3$, the following holds for $\Delta_\ell \in \N$ sufficiently large for each $2 \leq \ell \leq k$.
    Let $H$ be a triangle-free rank-$k$ hypergraph of maximum $\ell$-degree at most $\Delta_\ell$ for each $2 \leq \ell \leq k$.
    Then for some constant $c \defeq c(k) > 0$, we have
    \[\chi_\ell(H) \leq c\max_{2\leq \ell \leq k}\left(\frac{\Delta_\ell}{\log \Delta_\ell}\right)^{1/(\ell-1)}.\]
\end{theo}

In this paper we will be interested in $k$-partite $k$-graphs.
In fact, we establish a result similar to those of Theorems~\ref{theo:FM} and \ref{theo:LP} in the setting of list coloring.
Our main result is the following, which matches the bounds in Theorems~\ref{theo:FM} and \ref{theo:LP} asymptotically.

\begin{theo}\label{theo:main_theo}
    For all $\eps > 0$ and $k \geq 2$, the following holds for $\Delta$ sufficiently large.
    Let $H$ be a $k$-partite $k$-graph of maximum degree at most $\Delta$. 
    Then,
    \[\chi_\ell(H)\leq \left((k-1 + \eps)\frac{\Delta}{\log \Delta}\right)^{1/(k-1)}.\]
\end{theo}

In order to prove the above result, we establish a more general statement regarding list colorings.
Before we state the result formally, we make a few definitions regarding list assignments.
Let $H = (V, E)$ be a $k$-graph and let $L\,:\,V\to 2^\N$ be a list assignment for $H$.
We define:
\begin{align*}
    E_H(v, c) &\defeq \set{e\in E_H(v)\,:\,\forall u \in e,\,c\in L(u)}, \\
    \deg_H(v, c) &\defeq |E_H(v, c)|.
\end{align*}
We refer to $\deg_H(v, c)$ as the \textit{color-degree} of $c$ with respect to $v$.
A number of results mentioned earlier can be expressed in terms of the maximum color degree as opposed to the maximum degree.
The following result provides sufficient conditions on the maximum degree, maximum color degree, and list sizes with respect to a list assignment $L$ in order to construct an $L$-coloring for a $k$-partite $k$-graph.

\begin{theo}\label{theo:conditions_theo}
    For all $k \geq 2$, the following holds for $D_i, q_i, \Delta_i \in \N$ sufficiently large for each $i \in [k]$.
    Let $H$ be a $k$-partite $k$-graph with partition $V(H) = V_1\cup\cdots\cup V_k$, and let $L\,:\,V(H) \to 2^\N$ be a list assignment for $H$ such that the following hold for each $v \in V_i,\,c \in L(v)$:
    \begin{enumerate}[label=\ep{\normalfont{}\arabic*}]
        \item\label{item:degree_graph} $\deg_H(v) \leq \Delta_i$,
        \item\label{item:degree_cover} $\deg_H(v, c) \leq D_i$, and
        \item\label{item:list} $|L(v)| \geq q_i$.
    \end{enumerate}
    Suppose additionally, at least one of the following inequalities is satisfied for some $j \in [k]$:
    \begin{enumerate}[label=\ep{\normalfont{C}\arabic*}]
        \item\label{cond:1} $\prod_{i \in [k] - j}q_i \geq D_j\left(e\,q_j\sum_{i \in [k] - j}\frac{q_iD_i}{D_j}\right)^{1/q_j}$.
        \item\label{cond:3} $e\left(\Delta_j\left(\sum_{i \in [k]-j}\Delta_i - 1\right) + 1\right)\left(1 - \left(1 - \prod_{i \in [k] - j}q_i^{-1}\right)^{\Delta_j\,\min_{i \in [k]}q_i/q_j}\right)^{q_j} \leq 1$.
        \item\label{cond:2} $e\left(q_jD_j\left(\sum_{i \in [k] - j}q_iD_i - 1\right) + 1\right)\left(1 - \left(1 - \prod_{i \in [k] - j}q_i^{-1}\right)^{D_j}\right)^{q_j} \leq 1$.
    \end{enumerate}
    Then, $H$ admits a proper $L$-coloring.
\end{theo}

As corollaries, we get a few asymmetric list coloring results that we will state here and prove in \S\ref{sec:corollaries}.
First, we consider the case that the geometric mean of the maximum color degrees of the first $k-1$ partitions is much larger than that of the $k$-th.
This result matches that of Alon, Cambie, and Kang for $k = 2$ \cite[Corollary~8]{alon2021asymmetric}.

\begin{corl}\label{corl:d_to_eps}
    For all $\eps > 0$ and $k \geq 2$, the following holds for $D_i\in \N$ sufficiently large for each $i \in [k]$.
    Let $H$ be a $k$-partite $k$-graph with partition $V(H) = V_1\cup\cdots\cup V_k$, and let $L\,:\,V(H) \to 2^\N$ be a list assignment for $H$ such that the following hold for each $v \in V_i,\,c \in L(v)$:
    \begin{enumerate}[label=\ep{\normalfont{}\arabic*}]
        \item $\deg_H(v, c) \leq D_i$,
        \item $|L(v)| \geq D_i^{\eps/(k-1)}$, and
        \item\label{item:geom_bound} $\prod_{j \in [k-1]}D_j \geq D_k^{2(k-1)/\eps}$.
    \end{enumerate}
    Then, $H$ admits a proper $L$-coloring.
\end{corl}

Next, we consider the situation where most of the lists are ``small''.
The result matches that of \cite[Corollary~10]{alon2021asymmetric} for $k = 2$.

\begin{corl}\label{corl:const_k-1}
    For all $\eps > 0$ and $k \geq 2$, the following holds for $\Delta \in \N$ sufficiently large.
    Let $H$ be a $k$-partite $k$-graph with partition $V(H) = V_1\cup\cdots\cup V_k$ and maximum degree at most $\Delta$.
    Let $L\,:\,V(H) \to 2^\N$ be a list assignment for $H$ such that at least one of the following holds for $b \defeq \left(\frac{2^{k-1}}{2^{k-1} - 1}\right)^k$:
    \begin{enumerate}[label=\ep{\normalfont{}\arabic*}]
        \item\label{item:k} For each $v \in V(H)\setminus V_k$, $|L(v)| \geq 2$, and for each $v \in V_k$, $|L(v)| \geq \left(\frac{2+\eps}{k}\right)\Delta/\log_{b} \Delta$.
        \item\label{item:log} For each $v \in V(H)\setminus V_k$, $|L(v)| \geq \log \Delta$, and for each $v \in V_k$, $|L(v)| \geq \left(1+\eps\right)\Delta/(\log \Delta)^{k-1}$.
    \end{enumerate}
    Then, $H$ admits a proper $L$-coloring.
\end{corl}

Finally, we consider an asymmetric form of the main result of this paper, which matches that of Cambie and Kang for $k = 2$ \cite[Theorem~2.1]{cambie2022independent} (their proof holds in the more general setting of correspondence coloring defined below). 

\begin{corl}\label{theo:asym}
    For all $\eps > 0$ and $k \geq 2$, the following holds for $D_i$ sufficiently large for each $i \in [k]$.
    Let $H$ be a $k$-partite $k$-graph with partition $V(H) = V_1\cup\cdots\cup V_k$, and let $L\,:\,V(H) \to 2^\N$ be a list assignment for $H$ such that the following hold for each $v \in V_i,\,c \in L(v)$:
    \begin{enumerate}[label=\ep{\normalfont{}\arabic*}]
        \item\label{item:degree_cover} $\deg_H(v, c) \leq D_i$, and
        \item\label{item:list} $|L(v)| \geq \left((k-1 + \eps)\frac{D_i}{\log D_i}\right)^{1/(k-1)}$.
    \end{enumerate}
    Then, $H$ admits a proper $L$-coloring.
\end{corl}

By setting $D_i = D$ for each $i \in [k]$, we obtain the following immediate corollary, which is a ``color-degree'' version of Theorem~\ref{theo:main_theo} (setting $D_i = \Delta(H)$ yields Theorem~\ref{theo:main_theo}).

\begin{corl}\label{cor:main_result_color_degree}
    For all $\eps > 0$ and $k \geq 2$, the following holds for $D$ sufficiently large.
    Let $H$ be a $k$-partite $k$-graph, and let $L\,:\,V(H) \to 2^\N$ be a list assignment for $H$ such that the following hold for each $v \in V(H),\,c \in L(v)$:
    \[\deg_H(v, c) \leq D, \quad |L(v)| \geq \left((k-1 + \eps)\frac{D}{\log D}\right)^{1/(k-1)}.\]
    Then, $H$ admits a proper $L$-coloring.
\end{corl}

We conclude this section with a few remarks and potential directions for future research.
First, we note the following conjecture of Alon and Krivelevich:

\begin{conj}[{\cite[Conjecture~5.1]{alon1998choice}}]\label{conj:ak}
    For any bipartite graph $G$ of maximum degree at most $\Delta$, we have $\chi_\ell(G) = O(\log \Delta)$.
\end{conj}

In the same paper, Alon and Krivelevich show that the conjecture holds with high probability for Erd\H{o}s--R\'enyi random bipartite graphs.
In \cite{meroueh2019list}, the authors prove a similar result for Erd\H{o}s--R\'enyi random $k$-partite $k$-graphs $\mathcal{H}(k, n, p)$.
(Here, each partition has size $n$ and each valid edge is included independently with probability $p$, where an edge is valid if it contains exactly one endpoint in each partition.)

\begin{theo}[{\cite[Theorem~1.3]{meroueh2019list}}]
    Let $H \sim \mathcal{H}(k, n, p)$ such that $n^{k-1}p \geq \Delta_0$ for $\Delta_0$ large enough.
    Then, $\chi_\ell(H) = \Theta_{k}\left(\log(n^{k-1}p)\right)$ almost surely.
\end{theo}

In a similar flavor, Haxell and Verstraete showed that the complete $k$-partite $k$-graph with $n$ vertices in each partition (denoted $K_{k*n}$) satisfies $\chi_\ell(K_{k*n}) \leq (1+o_n(1))\log_kn$ \cite{haxell2010list}.
Note that $\Delta(H) \approx n^{k-1}p$ for $H \sim \mathcal{H}(k, n, p)$ with high probability, and $\Delta(K_{k*n}) = n^{k-1}$.
In particular, the above results can be expressed in terms of the maximum degree of the respective graphs.
In light of this, we make the following conjecture, which matches Conjecture~\ref{conj:ak} for $k = 2$.
(Note that Corollary~\ref{corl:const_k-1}\ref{item:log} provides partial progress toward it.)

\begin{conj}
    For all $k \geq 2$, there is a constant $c \defeq c(k) > 0$ such that the following holds for $\Delta$ large enough.
    Let $H$ be a $k$-partite $k$-graph of maximum degree at most $\Delta$.
    Then, we have $\chi_\ell(H) \leq c\log \Delta$.
\end{conj}

In \cite{cambie2022independent}, Cambie and Kang proved a result identical to Corollary~\ref{cor:main_result_color_degree} for $k = 2$ in the more general setting of \textit{correspondence coloring} also known as DP-coloring.
(In fact, their main result is a bipartite version of Theorem~\ref{theo:conditions_theo} for correspondence coloring.)
The concept was first introduced by Dvo\v{r}\'ak and Postle for graphs in \cite{DPCol}, while the extension to hypergraphs appeared in \cite{bernshteyn2019dp}.

\begin{defn}[Correspondence Cover]\label{def:corr_cov}
    A \emphd{correspondence cover} (also known as a \emphd{DP-cover}) of a $k$-graph $H$ is a pair $(L, \mathcal{H})$, where $\mathcal{H}$ is a $k$-graph and $L \,:\,V(H) \to 2^{V(\mathcal{H})}$ such that:
    \begin{itemize}
        \item The sets $L(v)\,:\,v\in V(H)$ partition $V(\mathcal{H})$,
        \item For each $v \in V(H)$ and $S \subseteq L(v)$ such that $|S| \geq 2$, $\deg_\mathcal{H}(S) = 0$, and
        \item For each $k$-element set $e \subseteq  V(H)$, the induced subgraph $\mathcal{H}\left[\bigcup_{v \in e}L(v)\right]$ is a hypermatching, which is empty if $e \notin E(H)$.
    \end{itemize}
\end{defn}

We call the vertices of $\mathcal{H}$ \emphd{colors}.
For $e \in E(\mathcal{H})$, we say that the colors in $e$ \emphd{correspond} to each other.
An \emphd{$(L,\mathcal{H})$-coloring} is a mapping $\phi \colon V(H) \to V(\mathcal{H})$ such that $\phi(v) \in L(v)$ for all $v \in V(H)$. 
An $(L,\mathcal{H})$-coloring $\phi$ is \emphd{proper} if the image of $\phi$ is an independent set in $\mathcal{H}$.
A correspondence cover $(L,\mathcal{H})$ is \emphdef{$q$-fold} if $|L(v)| \geq q$ for all $v \in V(H)$. 
The \emphdef{correspondence chromatic number} of $H$, denoted by $\chi_{c}(H)$, is the smallest $q$ such that $H$ admits a proper $(L,\mathcal{H})$-coloring with respect to every $q$-fold correspondence cover $(L,\mathcal{H})$.
A curious feature of correspondence covers for $k$-graphs satisfying $k \geq 3$ is that a color $c \in V(\mathcal{H})$ can correspond to different colors in the same list $L(v)$.
This causes our arguments for Theorem~\ref{theo:conditions_theo} to fail in this setting.
Nevertheless, we conjecture the bound in Corollary~\ref{cor:main_result_color_degree} holds for correspondence coloring as well.

\begin{conj}\label{conj}
    For all $\eps > 0$ and $k \geq 2$, the following holds for $D$ sufficiently large.
    Let $H$ be a $k$-partite $k$-graph, and let $(L,\mathcal{H})$ be a correspondence cover of $H$ such that the following hold for each $v \in V(H)$:
    \[\Delta(\mathcal{H}) \leq D, \quad |L(v)| \geq \left((k-1 + \eps)\frac{D}{\log D}\right)^{1/(k-1)}.\]
    Then, $H$ admits a proper $(L, \mathcal{H})$-coloring.
\end{conj}

Cambie and Kang conjectured that a similar result as theirs should hold for triangle-free graphs as most triangle-free graphs are close to bipartite.
However, we note that this property does not extend to $k$-partite $k$-graphs for $k \geq 3$.
In fact, such graphs need not be triangle-free, which makes it all the more surprising that our result matches that of Theorem~\ref{theo:LP}.
Furthermore, as noted in earlier work of the author \cite{corr}, the approach toward proving results on $k$-partite $k$-graphs is similar to those employed in other problems related to bipartite graphs.
It would be worth investigating when combinatorial results on bipartite graphs extend to the $k$-partite setting.

The rest of the paper is structured as follows.
In \S\ref{sec:prelim}, we will describe the probabilistic tools we will employ in our proofs.
In \S\ref{sec:corollaries}, we will show how the corollaries stated in this section follow from Theorem~\ref{theo:conditions_theo}, which we will prove in \S\ref{sec:proof}.

\section{Preliminaries}\label{sec:prelim}

In this section we describe probabilistic tools that will be used to prove Theorem~\ref{theo:conditions_theo}.
We start with the symmetric version of the Lov\'asz Local Lemma.

\begin{theo}[{Lov\'asz Local Lemma; \cite[Corollary~5.1.2]{AlonSpencer}}]\label{theo:LLL}
    Let $A_1$, $A_2$, \ldots, $A_n$ be events in a probability space. Suppose there exists $p \in [0, 1)$ such that for all $1 \leq i \leq n$ we have $\P[A_i] \leq p$. Further suppose that each $A_i$ is mutually independent from all but at most $d_{LLL}$ other events $A_j$, $j\neq i$ for some $d_{LLL} \in \N$. If $ep(d_{LLL}+1) \leq 1$, then with positive probability none of the events $A_1$, \ldots, $A_n$ occur.
\end{theo}

We will also need the following special case of the FKG inequality, dating back to Harris \cite{Harris} and Kleitman \cite{Kleitman}.
The original theorem is stated with regards two decreasing families, however, as the intersection of decreasing families is decreasing, it can be shown that the inequality holds in the following more general form.

\begin{theo}[{Harris's inequality/Kleitman's Lemma \cite[Theorem 6.3.2]{AlonSpencer}}]\label{theo:harris}
    Let $X$ be a finite set and let $S \subseteq X$ be a random subset of $X$ obtained by selecting each $x \in X$ independently with probability $p_x \in [0,1]$. If $\mathcal{A}_1, \ldots, \mathcal{A}_n$ are decreasing families of subsets of $X$, then 
    \[\P\left[S \in \bigcap_{i \in [n]}\mathcal{A}_i\right] \,\geq\, \prod_{i \in [n]}\P[S \in \mathcal{A}_i].\]
\end{theo}

\section{Proof of Corollaries}\label{sec:corollaries}

We may assume $|L(v)| = q_i$ for each $v \in V_i$ by arbitrarily removing colors from $L(v)$ if needed.
Let us show how each of the corollaries stated in \S\ref{sec:intro} follow from Theorem~\ref{theo:conditions_theo}.

\begin{proof}[Proof of Corollary~\ref{corl:d_to_eps}]
    We will show that \ref{cond:1} is satisfied.
    To this end, we note the following:
    \begin{align*}
        \prod_{i \in [k-1]}q_i &= \left(\prod_{i \in [k-1]}D_i^\eps\right)^{1/(k-1)} \\
        &\geq \,D_k^{1/2}\,\left(\prod_{i \in [k-1]}D_i^{\eps/4}\right)^{1/(k-1)}\,D_k
    \end{align*}
    It is easy to see the following for $D_k$ large enough:
    \[D_k^{1/2} \,\geq\, 1 \,\geq\, \left(eq_k/D_k\right)^{1/q_k}.\]
    Furthermore, assuming $D_k \geq (8(k-1)/\eps)^{(k-1)/\eps}$, we have
    \[\left(\prod_{i \in [k-1]}D_i^{\eps/4}\right)^{q_k/(k-1)} \,\geq\quad \prod_{i \in [k-1]}D_i^{2} \quad\gg\quad \sum_{i \in [k-1]}D_i^{1+\eps/(k-1)},\]
    completing the proof.
\end{proof}

\begin{proof}[Proof of Corollary~\ref{corl:const_k-1}]
    We will show that \ref{cond:3} is satisfied.
    Let us first consider \ref{item:k}.
    We have
    \begin{align*}
        \left(1 - \prod_{i \in [k-1]}q_i^{-1}\right)^{\Delta_j\,\min_{i \in [k]}q_i/q_k} &= \left(1 - \frac{1}{2^{k-1}}\right)^{\Delta\,q_1/q_k} \\
        &= \left(\frac{2^{k-1}}{2^{k-1} - 1}\right)^{-2\Delta\,\frac{k\log_b\Delta}{(2+\eps)\Delta}} \\
        &= \Delta^{-\frac{1}{1+\eps/2}} \\
        &\geq \Delta^{-(1 - \eps/4)}.
    \end{align*}
    Similarly, for \ref{item:log} we have
    \begin{align*}
        \left(1 - \prod_{i \in [k-1]}q_i^{-1}\right)^{\Delta_j\,\min_{i \in [k]}q_i/q_k} &= \left(1 - \frac{1}{(\log \Delta)^{k-1}}\right)^{\Delta\,q_1/q_k} \\
        &\geq \exp\left(-\frac{1}{(1-\eps/2)(\log\Delta)^{k-1}}\,\Delta\,q_1/q_k\right) \\
        &= \exp\left(-\frac{\log\Delta}{(1+\eps)(1-\eps/2)}\right) \\
        &\geq \Delta^{-(1-\eps/4)}.
    \end{align*}
    In either case, it follows that
    \begin{align*}
        \left(1 - \left(1 - \prod_{i \in [k-1]}q_i^{-1}\right)^{\Delta_k\,\min_{i \in [k-1]}q_i/q_k}\right)^{q_k} &\leq \exp\left(-q_k\,\Delta^{-(1-\eps/4)}\right) \\
        &\leq \exp\left(-\Delta^{\eps/10}\right).
    \end{align*}
    Since
    \[\left(\Delta_k\left(\sum_{i \in [k-1]}\Delta_i - 1\right) + 1\right) \quad\leq\quad 2k\Delta^2 \quad\ll\quad \exp\left(\Delta^{\eps/10}\right),\]
    condition \ref{cond:3} is satisfied.
\end{proof}

\begin{proof}[Proof of Corollary~\ref{theo:asym}]
    Without loss of generality, let $k = \arg\min_{j \in [k]}D_j$.
    We will consider two cases. 
    First, suppose $\prod_{i \in [k-1]}D_i \geq D_k^{2(k-1)/\eps}$.
    Then, the claim follows by Corollary~\ref{corl:d_to_eps} as $q_i \geq D_i^{\eps/(k-1)}$ for $D_i$ large enough.

    Now, suppose $\prod_{i \in [k-1]}D_i \leq D_k^{2(k-1)/\eps}$.
    Note the following: 
    \begin{align*}
        \left(1 - \prod_{i \in [k-1]}q_i^{-1}\right)^{D_k} &= \left(1 - \frac{1}{(k-1+\eps)}\left(\prod_{i \in [k-1]}\frac{\log D_i}{D_i}\right)^{1/(k-1)}\right)^{D_k} \\
        &\geq \exp\left(-\left(\frac{1 - \eps/(10k)}{(k-1+\eps)(1 - \eps/(4k))}\right)D_k\left(\prod_{i \in [k-1]}\frac{\log D_i}{D_i}\right)^{1/(k-1)}\right) \\
        &\geq \exp\left(-\left(\frac{1 - \eps/(10k)}{k-1}\right)\log D_k\right),
    \end{align*}
    where the last step follows since $D_i \geq D_k$ for each $i \in [k-1]$.
    From here, we can further simplify:
    \begin{align*}
        \left(1 - \left(1 - \prod_{i \in [k-1]}q_i^{-1}\right)^{D_k}\right)^{q_k} &\leq \exp\left(-q_k\,\exp\left(-\left(\frac{1 - \eps/(10k)}{k-1}\right)\log D_k\right)\right) \\
        &\leq \exp\left(-\left((k-1 + \eps)\frac{D_k^{\eps/(10k)}}{\log D_k}\right)^{1/(k-1)}\right) \\
        &\leq \exp\left(-D_k^{\eps/(20k^2)}\right).
    \end{align*}
    Since $\prod_{i \in [k-1]}D_i \leq D_k^{2(k-1)/\eps}$, we have
    \[q_jD_j \leq D_j^{k/(k-1)} \leq \prod_{i \in [k-1]}D_i^{k/k-1} \leq D_k^{2k/\eps}.\]
    In particular,
    \[q_kD_k\left(\sum_{i \in [k-1]}q_iD_i - 1\right) \quad\leq\quad kD_k^{3k/\eps} \quad\ll\quad \exp\left(D_k^{\eps/(20k^2)}\right).\]
    The claim now follows by \ref{cond:2}.
\end{proof}

\section{Proof of Theorem~\ref{theo:conditions_theo}}\label{sec:proof}

To prove Theorem~\ref{theo:conditions_theo}, we will construct a random partial coloring and show that it can be extended to the entire hypergraph.
Before we describe this procedure, we make the following definitions regarding a partial $L$-coloring $\phi\,:\,V(H) \pto \N$:
\begin{align*}
    \forall S \subseteq V(H),\, \phi(S) &\defeq \set{\phi(v) \,:\, v \in S}, \\
    \forall v \in V(H),\, L_\phi(v) &\defeq \set{c \in L(v)\,:\, \forall e \in E_H(v, c),\, \phi(e-c) \not = \set{c}}.
\end{align*}
In particular, $L_\phi(v)$ contains colors which may be assigned to $v$ to extend the coloring.
For each $v \in V(H) \setminus V_j$ we will independently pick $\phi(v) \in L(v)$ uniformly at random.
We will show that with positive probability, $L_\phi(v) \neq \0$ for each $v \in V_j$, completing the proof.
We will split this section into three subsections, containing the proofs under the conditions \ref{cond:1}, \ref{cond:3}, and \ref{cond:2}, respectively.

\subsection{Proof assuming \ref{cond:1}}\label{subsec:c1}

In order to prove Theorem~\ref{theo:conditions_theo} under the assumption of \ref{cond:1}, it is useful to consider the following correspondence cover of $H$:
\begin{itemize}
    \item Let $L$ be the list assignment of $H$ and let $\mathcal{H}$ be the cover graph whose vertices correspond to the colors in the lists defined by $L$.
    \item For each edge $\set{v_1, \ldots, v_k} \in E(H)$ and colors $c_i \in L(v_i)$, include the edge $\set{c_1, \ldots, c_k} \in E(\mathcal{H})$ if and only if $c_1 = \cdots = c_k$.
\end{itemize}
It can be verified that this defines a correspondence cover (also known as a \textit{list cover}) and that a proper $(L, \mathcal{H})$-coloring of $H$ is a proper $L$-coloring of $H$.
Before we begin the proof, we make a few definitions:
\begin{align*}
    \forall S \subseteq V(H),\, L(S) &\defeq \bigcup_{v \in S}L(v), \\
    \forall c \in V(\mathcal{H}),\, L^{-1}(c) &\defeq v \text{ such that } c \in L(v), \\
    \im(\phi) &\defeq \set{\phi(v)\,:\,v \in V(H)}, 
\end{align*}
Let us define an auxiliary hypergraph $\tilde H$ as follows:
\begin{itemize}
    \item $V(\tilde{\mathcal{H}}) \defeq V(\mathcal{H}) \setminus L(V_j)$.
    \item $S\subseteq V(\tilde{\mathcal{H}})$ forms an edge in $\tilde{\mathcal{H}}$ if the following hold:
    \begin{itemize}
        \item For all $c_1, c_2 \in S$, we have $L^{-1}(c_1) \neq L^{-1}(c_2)$, and
        \item there is a perfect hypermatching in $\mathcal{H}[S\cup L(v)]$ for some $v \in V_j$.
    \end{itemize}
\end{itemize}
We make the following observations about $\tilde{\mathcal{H}}$:
\begin{enumerate}[label=\ep{\normalfont{Obs}\arabic*}]
    \item\label{obs:size} As $\mathcal{H}[S\cup L(v)]$ contains a perfect hypermatching for some $v \in V_j$ and no two vertices in $S$ lie in the same list, $S$ must contain precisely $q_j$ vertices from $L(V_i)$ for each $i \in [k] - j$. 

    \item\label{obs:deg} For any $i \in [k]-j$ and $c \in L(V_i)$, $\deg_{\tilde{\mathcal{H}}}(c) \leq D_i\,D_j^{q_j - 1}$.
    This follows as there are at most $D_i$ choices for the matching edge $e$ containing $c$.
    From $e$, we may determine the vertex $v \in V_j$.
    For each remaining color $c' \in L(v)$, there are at most $D_j$ choices for the matching edge containing $c'$.
\end{enumerate}
For each edge $S \in E(\tilde{\mathcal{H}})$, let us define the following event:
\[A_S \defeq \bbone\set{S\subseteq \im(\phi)}.\]
If $A_S = 0$ for every $S \in E(\tilde{\mathcal{H}})$, then $L_\phi(v) \neq \0$ for each $v \in V_j$.
We note that this would not be the case for an arbitrary correspondence cover and this is where our argument fails in the DP-coloring setting.
From \ref{obs:size} and since no two vertices in $S$ lie in the same list, we have
\[\P[A_S = 1] \quad = \quad \prod_{c \in S}\frac{1}{|L(L^{-1}(c))|} \quad = \quad \left(\prod_{i \in [k]-j}q_i\right)^{-q_j}.\]
Let us now bound the number of events $A_{S'}$ such that $A_S$ is \textbf{not} mutually indpendent of $A_{S'}$.
The following is a valid upper bound as a result of \ref{obs:size} and \ref{obs:deg}:
\begin{align*}
    \sum_{c \in S}\sum_{c' \in L(L^{-1}(c))}\deg_{\tilde H}(c') - 1 &= \sum_{i \in [k] - j}\sum_{c \in S\cap L(V_i)}\sum_{c' \in L(L^{-1}(c))}\deg_{\tilde H}(c') - 1 \\
    &\leq \sum_{i \in [k] - j}\sum_{c \in S\cap L(V_i)}q_iD_i\,D_j^{q_j - 1} - 1 \\
    &= q_jD_j^{q_j - 1}\sum_{i \in [k] - j}q_iD_i - 1.
\end{align*}
We will apply the \hyperref[theo:LLL]{\LLL} with 
\[p \defeq \left(\prod_{i \in [k]-j}q_i\right)^{-q_j}, \quad d_{LLL} \defeq q_jD_j^{q_j - 1}\sum_{i \in [k] - j}q_iD_i - 1\]
to get:
\begin{align*}
    ep(d_{LLL} + 1) &= e\,\left(\prod_{i \in [k]-j}q_i\right)^{-q_j}\,\left(q_jD_j^{q_j - 1}\sum_{i \in [k] - j}q_iD_i\right) \\
    &\leq \frac{eq_jD_j^{q_j -1}\sum_{i \in [k] - j}q_iD_i}{\left(\prod_{i \in [k]-j}q_i\right)^{q_j}}.
\end{align*}
The above is at most $1$ as a result of \ref{cond:1}, completing the proof.

\subsection{Proof assuming \ref{cond:3}}

Recall that we define $\phi$ by randomly picking a color in $L(v)$ for each vertex $v \in V(H) \setminus V_j$.
The following lemma provides a bound on the probability $L_\phi(v) = \0$ for $v \in V_j$.

\begin{Lemma}\label{lemma:prob}
    For each $v \in V_j$, we have 
    \[\P[L_\phi(v) = \0] \leq \left(1 - \left(1 - \prod_{i \in [k] - j}q_i^{-1}\right)^{\sum_{c \in L(v)}\deg_H(v, c)/q_j}\right)^{q_j}.\]
\end{Lemma}

\begin{proof}
    For each $v \in V_j$ and $c \in L(v)$, define the following random variables:
    \[X_{v,c} \defeq \bbone\set{c\notin L_\phi(v)}, \quad X_v \defeq \prod_{c\in L(v)}X_{v, c}.\]
    It follows that $\P[L_\phi(v) = \0] = \Pr[X_v = 1]$.
    We will compute an upper bound for this probability through a series of claims.
    Let us first consider the event $\set{X_{v, c} = 1}$.
    
    \begin{claim}\label{claim:pr_color_lost}
        $\P[X_{v, c} = 1] \leq 1 - \left(1 - \prod_{i \in [k] - j}q_i^{-1}\right)^{\deg_H(v, c)}$.
    \end{claim}
    \begin{claimproof}
        We will lower bound $\P[X_{v, c} = 0]$ through \hyperref[theo:harris]{Harris's Inequality}.
        In order to do so, we define the following event for each $e \in E_H(v,c)$:
        \[X_e \defeq \bbone\set{\phi(e - v) = \set{c}}.\]
        Clearly, $\set{X_{v, c} = 0}$ is equivalent to $\set{\forall e \in E_H(v,c),\,X_e = 0}$.
        Let us define the set $\Gamma$ as follows:
        \[\Gamma \defeq \set{\set{\phi(u) = c}\,:\, u \in V(H), c \in L(u)},\]
        and let $S \subseteq \Gamma$ be the random events in $\Gamma$ that occur during our coloring procedure.
        Note that $S$ is formed by including each event $\set{\phi(u) = c}$ independently with probability $1/q_i$, where $u \in V_i$.
        Furthermore, $X_e = 1$ if and only if $\set{\phi(u) = c} \in S$ for each $u \in e-v$.
        Consider the following families for $e \in E_H(v,c)$:
        \[\mathcal{A}_e \defeq \set{S' \subseteq \Gamma\,:\, \text{when } S=S' \text{ we have } X_e = 0}.\]
        Let $S_1 \in \mathcal{A}_e$, and $S_2 \subseteq S_1$.
        Then, $S_2 \in \mathcal{A}_e$ as well.
        In particular, $\mathcal{A}_e$ is a decreasing family of subsets of $\Gamma$.
        Hence, by \hyperref[theo:harris]{Harris's Inequality}, we have
        \begin{align*}
            \P[X_{v, c} = 0] &\geq \prod_{e \in E_H(v,c)}\P[X_e = 0] \\
            &= \left(1 - \left(\prod_{i \in [k] - j}q_i\right)^{-1}\right)^{\deg_H(v,c)},
        \end{align*}
        as desired.
    \end{claimproof}

    We note that the above claim fails to hold in the DP-coloring setting and would be the main hurdle in proving Conjecture~\ref{conj}.
    In the next claim, we will show that the events $\set{X_{v, c} = 1}$ for $c \in L(v)$ are negatively correlated.

    \begin{claim}\label{claim:neg_corr}
        For every $I \subseteq L(v)$, we have 
        \[\P\left[\forall c \in I,\, X_{v, c} = 1\right] \leq \prod_{c \in I}\P\left[X_{v, c} = 1\right].\]
    \end{claim}
    \begin{claimproof}
        We will prove this by induction on $|I|$.
        The claim is trivial for $|I| \leq 1$.
        Suppose it holds for all $I$ such that $|I| = \ell$.
        Consider such a set $I$ and a color $c' \in L(v)\setminus I$.
        We note that 
        \[\P\left[\forall c \in I,\, X_{v, c} = 1\right] \leq \P\left[\forall c \in I,\, X_{v, c} = 1\mid X_{v, c'} = 0\right]\]
        as the probability to forbid all colors in $I$ is larger if no edge $e \in E_H(v)$ satisfies $\phi(e-v) = \set{c'}$.
        This is equivalent to:
        \begin{align*}
            &~\P\left[\forall c \in I,\, X_{v, c} = 1\right] \geq \P\left[\forall c \in I,\, X_{v, c} = 1\mid X_{v, c'} = 1\right] \\
            \iff&~\P\left[\forall c \in I\cup \set{c'},\, X_{v, c} = 1\right] \leq \P\left[\forall c \in I,\, X_{v, c} = 1\right] \P\left[X_{v, c'} = 1\right].
        \end{align*}
        By the induction hypothesis, the last expression is at most $\prod_{c \in I \cup \set{c'}}\P\left[X_{v, c} = 1\right]$.
    \end{claimproof}
    By Claims~\ref{claim:pr_color_lost} and \ref{claim:neg_corr}, we have
    \begin{align*}
        \P[X_v = 1] &\leq \prod_{c \in L(v)}\P[X_{v, c} = 1] \\
        &\leq \prod_{c \in L(v)}\left(1 - \left(1 - \prod_{i \in [k] - j}q_i^{-1}\right)^{\deg_H(v, c)}\right).
    \end{align*}
    Note that the function $\log (1 - c^x)$ is concave and increasing for $0 < c < 1$.
    Therefore, we have the following by Jensen's inequality
    \[\P[X_v = 1] \leq \left(1 - \left(1 - \prod_{i \in [k] - j}q_i^{-1}\right)^{\sum_{c \in L(v)}\deg_H(v, c)/q_j}\right)^{q_j},\]
    completing the proof.
\end{proof}

Note that
\[\sum_{c\in L(v)}\deg_H(v,c) \,=\sum_{e \in E_H(v)}\left|\bigcap_{u \in e}L(u)\right| \quad\leq\quad \deg_H(v)\,\min_{i \in [k]}q_i \quad\leq\quad \Delta_j\,\min_{i \in [k]}q_i.\]
Plugging this into the result of Lemma~\ref{lemma:prob}, we get
\[\P[L_\phi(v) = \0] \leq \left(1 - \left(1 - \prod_{i \in [k]}q_i^{-1}\right)^{\Delta_j\,\min_{i \in [k]}q_i/q_j}\right)^{q_j}.\]
For each $v \in V_j$, let $A_v$ denote the event that $L_\phi(v) = \0$.
The goal is to show that none of these events occur with positive probability.
Note that $A_v$ is mutually independent of all but at most $\Delta_j\left(\sum_{i \in [k]-j}\Delta_i - 1\right)$ events $A_u$.
We will apply the \hyperref[theo:LLL]{\LLL} with 
\[p \defeq \left(1 - \left(1 - \prod_{i \in [k] - j}q_i^{-1}\right)^{\Delta_j\,\min_{i \in [k]}q_i/q_j}\right)^{q_j}, \quad d_{LLL} \defeq \Delta_j\left(\sum_{i \in [k] - j}\Delta_i - 1\right)\]
to get:
\begin{align*}
    ep(d_{LLL} + 1) &= e\left(\Delta_j\left(\sum_{i \in [k]-j}\Delta_i - 1\right) + 1\right)\left(1 - \left(1 - \prod_{i \in [k] - j}q_i^{-1}\right)^{\Delta_j\,\min_{i \in [k]}q_i/q_j}\right)^{q_j}
\end{align*}
The above is at most $1$ as a result of \ref{cond:3}, completing the proof.

\subsection{Proof assuming \ref{cond:2}}

By Lemma~\ref{lemma:prob} and since $\deg_H(v, c) \leq D_j$, we have
\[\P[L_\phi(v) = \0] \leq \left(1 - \left(1 - \prod_{i \in [k]}q_i^{-1}\right)^{D_j}\right)^{q_j}.\]
For each $v \in V_j$, let $A_v$ denote the event that $L_\phi(v) = \0$.
The goal is to show that none of these events occur with positive probability.
Note that $A_v$ is mutually independent of all but at most $q_jD_j\left(\sum_{i \in [k] - j}q_iD_i - 1\right)$ events $A_u$.
We will apply the \hyperref[theo:LLL]{\LLL} with 
\[p \defeq \left(1 - \left(1 - \prod_{i \in [k] - j}q_i^{-1}\right)^{D_j}\right)^{q_j}, \quad d_{LLL} \defeq q_jD_j\left(\sum_{i \in [k] - j}q_iD_i - 1\right)\]
to get:
\begin{align*}
    ep(d_{LLL} + 1) &= e\,\left(q_jD_j\left(\sum_{i \in [k] - j}q_iD_i - 1\right) + 1\right)\left(1 - \left(1 - \prod_{i \in [k] - j}q_i^{-1}\right)^{D_j}\right)^{q_j}.
\end{align*}
The above is at most $1$ as a result of \ref{cond:2}, completing the proof.

\subsection*{Acknowledgments}

We thank the anonymous referees for their helpful comments.

\vspace{0.1in}
\printbibliography

\end{document}